\documentclass[12pt]{article}
\usepackage[dvipdfmx]{graphicx}
\usepackage[margin=30truemm]{geometry}
\usepackage{amsmath}
\usepackage{amssymb}
\usepackage{amsthm}
\usepackage{authblk}
\usepackage{mathrsfs}
\usepackage{slashed}
\usepackage{color}
\usepackage{cite}
\usepackage{here}
\usepackage{subcaption}
\usepackage[normalem]{ulem}

\makeatletter
 
 \@addtoreset{equation}{section}
 \makeatother

\newcommand{\be}{\begin{equation}}
\newcommand{\ee}{\end{equation}}

\newcommand{\bea}{\begin{equnarray}}
\newcommand{\eea}{\end{eqnarray}}
\newcommand{\beann}{\begin{equnarray*}}

\newcommand{\eeann}{\end{eqnarray*}}
\newcommand{\nn}{\nonumber}
\newcommand{\ba}{\begin{array}}
\newcommand{\ea}{\end{array}}

\newcommand{\bs}{\boldsymbol}

\newcommand{\bse}{{\boldsymbol{e}}}
\newcommand{\bsf}{{\boldsymbol{f}}}
\newcommand{\bsone}{\bs{1}}

\newcommand{\N}{\mathbb N}

\newcommand{\R}{\mathbb R}

\newcommand{\tlt}{{\tilde{t}}}

\newcommand{\tL}{{\tilde{L}}}

\newcommand{\tDelta}{{\tilde{\Delta}}}

\newcommand{\adjSigma}{\Sigma_{\rm adj}}
\newcommand{\chQ}{{\check{Q}}}

\newcommand\dashint{\mathchoice
  {\int\kern-10pt-}
  {\int\kern-8.5pt-}
  {\int\kern-6.1pt-}
  {\int\kern-4.58pt-}}

\newcommand{\ds}{\displaystyle}

\DeclareMathOperator{\rank}{rank}
\DeclareMathOperator{\diag}{diag}

\DeclareMathOperator{\Tr}{Tr}
\DeclareMathOperator{\Spec}{Spec}

\DeclareMathOperator{\re}{Re}

\newtheorem{theorem}{Theorem}
\newtheorem{lemma}{Lemma}
\newtheorem{corollary}{Corollary}
\newtheorem{prop}{Proposition}

\newcommand{\mathscale}[2]{\text{\scalebox{#1}{${#2}$}}}

\title{Functional Equations and Pole Structure of the Bartholdi Zeta Function}
\author[1,2]{So Matsuura\thanks{s.matsu@keio.jp}}
\author[2]{Kazutoshi Ohta\thanks{kazutoshi.ohta@mi.meijigakuin.ac.jp}}
\affil[1]{\it Hiyoshi Departments of Physics,
and Research and Education Center for Natural Sciences,
Keio University, 
Yokohama, Kanagawa 223-8521, Japan}
\affil[2]{\it Institute for Mathematical Informatics,
Meiji Gakuin University, Yokohama, Kanagawa 244-8539, Japan}

\date{}

\begin{document}
\maketitle


\begin{center}
{\bf Abstract}
\end{center}
In this paper, 
we investigate the Bartholdi zeta function on a connected simple digraph with $n_V$ vertices and $n_E$ edges. 
We derive a functional equation for the Bartholdi zeta function $\zeta_G(q,u)$ on a regular graph $G$ with respect to the bump parameter $u$.
We also find an equivalence between the Bartholdi zeta function with a specific value of $u$ and the Ihara zeta function at $u=0$.
We determine bounds of the critical strip of $\zeta_G(q,u)$ for a general graph. 
If $G$ is a $(t+1)$-regular graph, the bounds are saturated and $q=(1-u)^{-1}$ and $q=(t+u)^{-1}$ are the poles at the boundaries of the critical strip for $u\ne 1, -t$.
When $G$ is the regular graph and the spectrum of the adjacency matrix satisfies a certain condition, 
$\zeta_G(q,u)$ satisfies the so-called Riemann hypothesis. 
For $u \ne 1$, $q=\pm(1-u)^{-1}$ are poles of $\zeta_G(q,u)$ unless $G$ is tree.
Although the order of the pole at $q=(1-u)^{-1}$ is $n_E-n_V+1$ if $u\ne u_* \equiv 1-\frac{n_E}{n_V}$, 
it is enhanced at $u=u_*$. 
In particular, 
if the Moore-Penrose inverse of the incidence matrix $L^+$ and the degree vector $\vec{d}$ satisfy the condition $|L^+ \vec{d}|^2\ne n_E$, the order of the pole at $q=(1-u)^{-1}$ increases only by one at $u=u_*$.
The order of the pole at $q=-(1-u)^{-1}$ coincides with that at $q=(1-u)^{-1}$ if $G$ is bipartite and is $n_E-n_V$ otherwise. 

\newpage

\section{Introduction and Conclusion}
\label{sec:intro}

Graph zeta functions \cite{Ihara:original,MR607504,sunada1986functions,Hashimoto1990ONZA,bass1992ihara,bartholdi2000counting,mizuno2003bartholdi},
initially introduced as an analog to the Selberg zeta function, have emerged as a significant tool primarily in mathematics. 
For instance, in algebraic graph theory, they are instrumental in analyzing graph invariants and characterizing graph isomorphisms. 
In the realm of combinatorics, graph zeta functions help in analyzing the complexity of graph structures and in deriving enumerative properties of walks on graphs. Furthermore, they play a significant role in number theory, particularly in the examination of arithmetic properties of graphs and their connections to modular forms \cite{sugiyama2017}
\footnote{
  For comprehensive reviews and references of the graph zeta functions, see also \cite{terras_2010}.
}.

Recently, 
an interesting relationship between the graph zeta functions and gauge theories on the discrete spacetime has been discovered in
\cite{matsuura2022kazakov, matsuura2022graph,PhysRevD.108.054504,Matsuura:2024gdu}, 
where
the so-called Kazakov-Migdal (KM) model \cite{kazakov1993induced} has been generalized on the graph $G$. 
A characteristic feature of this generalized KM model is that the partition function is expressed as an integral of a matrix weighted graph zeta function over the unitary matrices,
\be
Z = {\cal N}\int \prod_{e}d U_e \,
\zeta_G(q,u;U)\,,
\ee
where 
${\cal N}$ is a normalization constant and $\zeta_G(q,u; U)$ is the Bartholdi zeta function weighted by a specific representation of the unitary group $U_e$ on the edge $e$. 
This enables the generalized KM model to accurately enumerate Wilson loops on the graph and to precisely evaluate the partition function in the large size limit of the unitary matrix.

In the model studied in \cite{matsuura2022kazakov, matsuura2022graph}, the unitary matrices $U_e$ are chosen in the adjoint representation, 
which is a straightforward extension of the original KM model. 
While this model displays fascinating characteristics from a mathematical point of view, 
its application as a model of QCD in physics, which was the foundational motivation behind the KM model, is faced with the same difficulties as the original KM model due to the presence of an undesirable local $U(1)$ symmetry \cite{kogan1992induced,kogan1993area,kogan1993continuum, migdal1993bose,cline1993induced,balakrishna1994difficulties}.
In \cite{PhysRevD.108.054504}, this problem is resolved by choosing the representation of the unitary matrices in the fundamental representation.
We call the model the fundamental Kazakov-Migdal (FKM) model.

As a result, the effective action of the FKM model is expressed in terms of a sum of all possible Wilson loops on the graph in the fundamental representation.
The effective action also contains the usual Wilson action of lattice gauge theory in a certain parameter limit.
Therefore, we can expect that the FKM model induces a realistic gauge theory in an appropriate limit of the parameters.
In \cite{PhysRevD.108.054504}, 
it was also shown that the FKM model undergoes the so-called Gross-Witten-Wadia (GWW) phase transition \cite{Gross:1980he,Wadia:1980cp},
regardless of the details of the graph. 
This suggests that the GWW phase transition is a quite universal phenomenon of QCD-type gauge theories.

In \cite{Matsuura:2024gdu},  
using the properties of the Ihara zeta function extensively,
the nature of the FKM model has been examined in detail in the case where the partition function is expressed by the
integral of the matrix weighted Ihara zeta function. 
It was found that the FKM model on a regular graph exhibits a duality between large and small values of $q$ as a consequence of the functional equation of the Ihara zeta function. 
Moreover, it was revealed that the FKM model shows instability when $q$ is in the so-called critical strip of the Ihara zeta function. 
The critical strip corresponds to the distribution of the poles of the Ihara zeta function.
This suggests that the pole structure of the graph zeta function in the complex plane is strongly related to the non-perturbative properties of the FKM model. 

Apart from the $u=0$ case studied in \cite{Matsuura:2024gdu},
we can also expect that the FKM model has richer structure in the parameter space of $q$ and $u$, where the partition function is expressed by the Bartholdi zeta function.
However, despite the extensive study of the Ihara zeta function (see e.g.~Ch.~7 in the book \cite{terras_2010}
and references therein) 
and the Bartholdi zeta function \cite{guido2008bartholdi},
details of the Bartholdi zeta function such as the functional equation with respective to the bump parameter and pole structure cannot be found in the literature.
Therefore, we are motivated to investigate deeper properties of Bartholdi zeta function mathematically in this paper,
prior to applying it to the physical model.
These properties will not only provide a basis for examining the FKM model, but also reveal non-trivial relationships among the graph zeta functions.

In this paper, we consider only connected simple directed graphs. 
The main results are the following: 
\begin{itemize}
  \item The Bartholdi zeta function $\zeta_G(q,u)$ on a regular graph $G$ 
  satisfies a functional equation 
  with respect to the bump parameter $u$ (Theorem \ref{th:funcional eq for u}). 
  \item The Bartholdi zeta function $\zeta_G(q,u)$ on a $(t+1)$-regular graph $G$ with $u=1-t$ is equivalent to the Ihara zeta function $\zeta_G(q)$ on the same graph up to an irrelevant factor (Corollary \ref{col:Ihara-Bartholdi duality}). 
  \item There are bounds of the boundaries of the critical strip of $\zeta_G(q,u)$ on a general graph $G$ (Proposition \ref{prop:range of poles}). 
  \item $\zeta_G(q,u)$ on a general graph $G$ has a simple pole at the lower boundary of the critical strip of $|q|$ if $u\ge 0$ and has a simple pole at the upper boundary if $u\le -\tlt_M+1$ where $\tlt_M+1$ is the maximal degree of the vertices in $G$ (Theorem \ref{th:first pole at boundary}).
  \item Suppose $G$ is a $(t+1)$-regular graph and $u\ne 1,-t$. 
  Then, $q=(1-u)^{-1}$ and $q=(t+u)^{-1}$ are poles at the boundaries of the critical strip which saturate the bounds given in Proposition \ref{prop:range of poles} (Theorem \ref{th:range of poles regular}). 
  \item When the spectrum of the adjacency matrix satisfies a specific condition, the Bartholdi zeta function on a regular graph satisfies the so-called Riemann hypothesis (Theorem \ref{th:Riemann hypothesis Bartholdi}). 
  \item Suppose $u\ne 1$. If $G$ is not a tree graph, $q=(1-u)^{-1}$ is a pole of $\zeta_G(q,u)$ of order $n_E-n_V+1$ when $u\ne u_* \equiv 1-\frac{n_E}{n_V}$ while 
  the order is enhanced to more than or equal to $n_E-n_V+2$ when and only when $u=u_*$. 
  If $G$ is a tree graph, $q=(1-u)^{-1}$ is not a pole when $u\ne u_*$ and it becomes a pole when and only when $u=u_*$.  
  In particular, if the Moore-Penrose inverse of the incidence matrix $L^{+}$ and the degree vector $\vec{d}$ satisfy $|L^+ \vec{d}|^2\ne n_E$, the enhanced degree of the pole at $u=u_*$ is exactly $n_E-n_V+2$ (Theorem \ref{th:enhancement}). 
  \item Suppose $u\ne 1$. 
  When $G$ is not bipartite, 
  $q=-(1-u)^{-1}$ is a pole of $\zeta_G(q,u)$ of order $n_E-n_V$ if $n_E>n_V$ 
  and is not a pole if $n_E=n_V$. 
  When $G$ is bipartite, 
  the bahavior of $\zeta_G(q,u)$ at $q=-(1-u)^{-1}$ is the same as that at $q=(1-u)^{-1}$. 
  (Corollary \ref{col:pole at minus 1-u inv}). 
\end{itemize}
Note that, if $G$ is a regular graph, $u=u_*$ coincide with the critical value of the functional equation of the Bartholdi zeta function. 
This implies that the Bartholdi zeta function on an irregular graph also has a comprehensive relation among its parameters, and that the functional equations are a manifestation of that relation in a regular graph.

The organization of this paper is as follows: 
In the next section, we summarize notations of the graph theory and introduce the Ihara and Bartholdi zeta functions as preliminaries. 
In Sec.~\ref{sec:functional equations}, 
we discuss the functional equations of the Bartholdi zeta function. 
In Sec.~\ref{sec:poles}, 
we discuss the pole structure of the Bartholdi zeta function. 
The final section is devoted to discussions and future problems.

\section{Preliminaries}
\label{sec:preliminaries}

\subsection{Graphs, paths and cycles} 
\label{subsec:notations}
A graph $G$ is an object which consists of
vertices and edges that connect two vertices. 
We denote the set of the vertices and edges by $V$ and
$E$, respectively,
with their cardinalities represented by $n_V$ and $n_E$. 
We assume that the edges have directions. 
Therefore, the term {\it graph} in this paper refers to a directed graph (digraph). 
A directed edge $e$ starting from a vertex $v$ and ending at a vertex $v'$ is represented by $e=\langle v,v'\rangle$, where
the two vertices $v$ and $v'$ 
are called 
the {\it source} $v=s(e)$ and the {\it target} $v'=t(e)$ of $e$, respectively.
The degree of a vertex $v$, denoted as $\deg v$, is the number of its neighboring vertices. 
It is useful to consider 
both a directed edge $e=\langle v,v'\rangle$ and its opposite $e^{-1}=\langle v',v\rangle$ at the same time. 
We then enhance the set of the edges $E$ to
$E_D=\{\bse_a|a=1,\cdots,2n_E\}
=\{e_1,\cdots, e_{n_E}, e_1^{-1}, \cdots e_{n_E}^{-1}\}$,
combining with the opposite directed edges.
In the following, we assume that there is at most one edge connecting two vertices, and  no edge connecting the same vertex. In other words, we consider only simple digraphs in this paper.

A path (or walk) of length $l$ on the graph $G$ is a sequence of $l$ edges, $P=\bse_{a_1} \bse_{a_2} \cdots \bse_{a_l}$ satisfying the conditions $t(\bse_{a_i})=s(\bse_{a_{i+1}})$ ($i=1,\cdots,l-1$). 
The inverse of the path $P$ is defined as $P^{-1}=\bse_{a_l}^{-1}\cdots \bse_{a_1}^{-1}$. 
A path $C$ whose ending vertex coincides with the starting one
is called a cycle.
We can define the product of two paths $P_1=\bse_{a_1} \cdots \bse_{a_{l_1}}$ and $P=\bse_{a'_1} \cdots \bse_{a'_{l_2}}$ when $t(\bse_{a_l})=s(\bse_{a'_1})$ as $P_1P_2 =\bse_{a_1} \cdots \bse_{a_{l_1}}\bse_{a'_1} \cdots \bse_{a'_{l_2}}$. 
In particular, a power of a cycle $C^n$, which is constructed through this product rule, becomes a cycle again. 
In this paper, we consider only connected graphs, that is, graphs where any pair of vertices is connected by some path.

A part of a path of length $l$ satisfying $\bse_{a_{j+1}}=\bse_{a_{j}}^{-1}$ ($j=1,\ldots, l-1$) is called a backtracking. 
If a cycle of length $l$ satisfies $\bse_{a_l}=\bse_{a_1}^{-1}$, this part is called a tail of the cycle $C$. 
For a cycle $C$, the equivalence class $[C]$ is defined by the set of cyclic rotations of $C$ as 
\[[C] = \{\bse_{a_1} \bse_{a_2} \cdots \bse_{a_l}, \,
\bse_{a_2} \cdots \bse_{a_l}\bse_{a_1}, \, \ldots \,,
\bse_{a_l}\bse_{a_1}\cdots\bse_{a_{l-1}}\}\,. \]
For the equivalence class of the cycle $[C]$, the backtracking and the tail are identical, hence they are collectively called a bump.

A cycle $C$ is called reduced when it does not contain the bump. 
A cycle $C$ is called primitive when $C$ does not satisfy $C \ne (C')^r$ for any cycle $C'$ and $r\ge 2$.
We denote the set of representatives of reduced cycles as $[{\mathcal P}_R] \subset [{\mathcal P}]$.
Since a primitive reduced cycle $C$ has its inverse $C^{-1}$ also be a primitive reduced cycle of equal length, 
the set $[{\mathcal P}_R]$ can be partitioned into two disjoint unions;  
$[{\cal P}_R]=[\Pi_{+}] \sqcup [\Pi_{-}]$, 
where $[\Pi_{-}]$ consists of the inverses of elements in $[\Pi_{+}]$. These elements in $[\Pi_{+}]$ are referred to as the representatives of chiral primitive reduced cycles.

\subsection{Matrices associated with the graph}
\label{subsec:matrices}

We can define several matrices associated with a graph $G$. 

The adjacency matrix $A$ is a symmetric matrix of size $n_V$ whose element $A_{vv'}$ ($v,v'\in V$) is either $1$ or $0$ depending on whether $v'$ is a neighbor of $v$ or not;
\begin{equation}
  A_{vv'} \equiv \sum_{\bse\in E_D}
  \delta_{\langle{v},{v'}\rangle,\bse} \,.  \quad (v,v'\in V)
  \label{eq:matrix A}
\end{equation}

The degree matrix $D$ is a diagonal matrix of size $n_V$ defined as
\begin{equation}
D\equiv {\rm diag}(\deg v_1,\deg v_2, \cdots, \deg v_{n_V})\,.
\label{eq:matrix D}
\end{equation}
Note that the diagonal part of the matrix $A^2$ is equal to $D$. 
In the following, we sometimes abbreviate $\deg v_i$ ($i=1,\cdots,n_V$) as 
\begin{equation}
  d_i \equiv \deg v_i\,. \quad (i=1,\cdots,n_V)
\end{equation}

Corresponding to the degree matrix $D$, we define the matrix $Q$ as 
\begin{equation}
Q\equiv 
D - \bsone_{n_V}\,,
\label{eq:matrix Q}
\end{equation}
and its parameter deformation by $u$ as 
\begin{equation}
Q_u \equiv 
(1-u)\left(D-(1-u)\bsone_{n_V}\right)\,.
\label{eq:matrix Qu}
\end{equation}
Note that $Q_{u=0} = Q$. 

The incidence matrix $L$ 
and the $(0,1)$-incidence matrix $\tL$ 
are $n_E\times n_V$ matrices whose elements are given by 
\begin{equation}
  L_{ev} \equiv 
  \begin{cases}
    1 & v = t(e) \\
    -1 & v = s(e) \\
    0 & {\rm otherwise}
  \end{cases}\,,
  \label{eq:incidence matrix}
\end{equation}
and 
\begin{equation}
  \tL_{ev} \equiv 
  \begin{cases}
    1 & v = s(e)\ {\rm or}\ v=t(e) \\
    0 & {\rm otherwise}
  \end{cases}\,,
  \label{eq:01incidence matrix}
\end{equation}
respectively, where $e\in E$ and $v\in V$. 
As seen from the definitions, the incidence matrix and the $(0,1)$-incidence matrix give relations between the edges and the vertices. 

An important matrix made of $A$ and $D$ is the graph Laplacian, 
\begin{equation}
  \Delta \equiv D-A\,,
  \label{eq:graph Laplacian}
\end{equation}
which is related to the incidence matrix as 
\begin{equation}
  \Delta = L^T L = D - A\,.
  \label{eq:incidence and Laplacian}
\end{equation}
For a later purpose, we also define 
\begin{equation}
  \tDelta \equiv \tL^T \tL = D + A\,.
  \label{eq:01incidence and DA}
\end{equation}

The edge adjacency matrix is a matrix of size $2n_E$ whose element $W_{\bse\bse'}$ ($\bse,\bse'\in E_D$) is either $1$ or $0$ corresponding to whether the conditions $t(\bse)=s(\bse')$ and $\bse'\ne\bse^{-1}$ are satisfied or not: 
\begin{align}
  W_{\bse\bse'} \equiv \begin{cases}
    1 & {\rm if}\ t(\bse) = s(\bse')\ {\rm and}\ \bse'^{-1}\ne \bse \\
    0 & {\rm others}
  \end{cases}\,.
  \quad \left(\bse,\bse'\in E_D\right)
  \label{eq:matrix W}
\end{align}
The bump matrix $J$ is a matrix of size $2n_E$ whose element $J_{\bse\bse'}$ ($\bse,\bse'\in E_D$) is either $1$ or $0$ corresponding to whether $\bse'$ is the inverse of $\bse$ or not: 
\begin{align}
  J_{\bse\bse'} \equiv \begin{cases}
    1 & {\rm if}\ \bse'^{-1}= \bse \\
    0 & {\rm others}
  \end{cases}\,.
  \quad \left(\bse,\bse'\in E_D\right)
  \label{eq:matrix J}
\end{align}
Similar to the degree matrix $D$,
we also define a diagonal degree matrix of size $2n_E$
associated with the above edge representation by
\be
\check{D} \equiv {\rm diag}\left(
  \deg s(\bse_1),\deg s(\bse_2), \cdots, \deg s(\bse_{2n_E})
\right)\,.
\ee
Using this, as an analog to \eqref{eq:matrix Qu},
we define a diagonal matrix of size $2n_E$ by
\begin{equation}
  \check{Q}_u \equiv (1-u)\left(\check{D}-(1-u)\bsone_{2n_E}\right)\,.
\end{equation}

\subsection{The Ihara and Bartholdi zeta functions}
\label{subsec:zeta functions}

The Bartholdi zeta function \cite{bartholdi2000counting} is defined by the Euler product over the representatives of primitive cycles,
\begin{equation}
  \zeta_G(q,u) \equiv \prod_{C\in [{\cal P}]} \frac{1}{1-q^{|C|} u^{b(C)}}\,,
  \label{eq:Bartholdi}
\end{equation}
where $|C|$ and $b(C)$ denote the length and the number of bumps of the cycle $C$, respectively. 
We can rewrite this expression as 
\begin{align}
  \zeta_G(q,u) = \exp\left(
  \sum_{C:\text{all cycles}}\frac{1}{|C|}q^{|C|}u^{b(C)}
  \right)\,,
  \label{eq:Bartholdi exp}
\end{align}
where the summation runs over all possible cycles of $G$. 
In order for the series expansion to converge, $|q|$ and $|u|$ must be sufficiently small. 

The Ihara zeta function \cite{Ihara:original,MR607504,sunada1986functions} is obtained from the Bartholdi zeta function by setting $u=0$, 
\begin{equation}
  \zeta_G(q)\equiv
 \zeta_G(q,u=0)
 = \prod_{C\in [{\cal P}_R]} \frac{1}{1-q^{|C|} }\,,
  \label{eq:Ihara}
\end{equation}
where the range of the product is reduced to the representatives of the primitive reduced cycles, 
which can be rewritten as 
\begin{align}
\zeta_G(q) = \exp\left(
\sum_{m=1}^\infty \frac{N_m}{m} q^{m}
\right)\,,
\end{align}
where $N_m$ stands for the number of reduced cycles of length $m$. 

Although the Euler product \eqref{eq:Bartholdi} (and \eqref{eq:Bartholdi exp}) converges only for sufficiently small $|q|$ and $|u|$, 
the Bartholdi zeta function (and the Ihara zeta function as well) is analytically connected to the whole complex plane in the form of the reciprocal of the polynomial of $q$ and $u$.
It is achieved by the fact that the Bartholdi zeta function \eqref{eq:Bartholdi} is expressed as the inverse of the determinant of a matrix\cite{bartholdi2000counting};
  \begin{equation}
    \zeta_G(q,u) 
    = \bigl(1-(1-u)^2q^2\bigr)^{-(n_E-n_V)}
    \det\bigl(\bsone_{n_V}-q A + q^2 Q_u \bigr)^{-1}\,,
    \label{eq:vertex Bartholdi}
  \end{equation}
  where $A$ and $Q_u$ are given by \eqref{eq:matrix A} and \eqref{eq:matrix Qu}, respectively. 
This is called the vertex expression of the Bartholdi zeta function. 
In particular, if $G$ is the regular graph, it reduces to
\begin{equation}
  \zeta_G(q,u) = \left(1-(1-u)^2q^2\right)^{-(n_E-n_V)}
    \det\left((1+(1-u)(t+u)q^2)\bsone_{n_V}-q A \right)^{-1}\,.
    \label{eq:vertex Bartholdi regular}
\end{equation}

By setting $u=0$, the Bartholdi zeta function reduces to
the vertex expression of the Ihara zeta function
  \begin{equation}
    \begin{split}
    \zeta_G(q) 
    &= \bigl(1-q^2\bigr)^{-(n_E-n_V)}
    \det\bigl(\bsone_{n_V}-q A + q^2 Q\bigr)^{-1} \,. 
    \end{split}
    \label{eq:vertex Ihara}
  \end{equation}

Apart from the vertex expression, another expression focusing on the relation among the edges of $G$ is known \cite{mizuno2003bartholdi};
  \begin{equation}
    \zeta_G(q,u) = \det \left(\bsone_{2n_E}- q(W+uJ)\right)^{-1}\,, 
    \label{eq:edge Bartholdi}
  \end{equation}
  where $W$ and $J$ are the edge adjacency matrix \eqref{eq:matrix W} and the bump matrix \eqref{eq:matrix J}, respectively. 
This expression is called the Hashimoto expression \cite{Hashimoto1990ONZA} or the edge expression of the Bartholdi zeta function. 
Again, by setting $u=0$, we obtain the edge expression of the Ihara zeta function.
  The Ihara zeta function can be expressed through the edge adjacency matrix as 
  \begin{equation}
    \zeta_G(q) = \det \left(\bsone_{2n_E}- q W\right)^{-1}\,.
    \label{eq:edge Ihara}
  \end{equation}
Note that the equivalence between the vertex expression \eqref{eq:vertex Ihara} and the edge expression \eqref{eq:edge Ihara} was shown in a clear manner in \cite{bass1992ihara}.

\section{Functional equations for the Bartholdi zeta function}
\label{sec:functional equations}

In this section, we discuss functional equations
that relate the Bartholdi zeta function $\zeta_G(q,u)$ with different values of $q$ and $u$ to each other.

\subsection{The functional equation with respect to $q$} 
\label{subsec:duality on irregular graph}

For the following convenience, we first introduce 
\begin{equation}
  t_{v} \equiv \deg v-1\,. \quad (v\in V)
\end{equation}
Then, let us start with the following lemma: 

\begin{lemma}
\label{lem:general duality vertex}
If $u\ne 1$ nor $u\ne -t_v$ for ${}^\forall v\in V$, namely, the matrix $Q_u$ is invertible, 
the Bartholdi zeta function on a general graph $G$ has the following expression for $c\ne 0$, 
\begin{align}
  \zeta_G(1/cq,u) 
  = 
  \frac{1}{\prod_{v\in V}(1-u)(t_v+u)} 
  &\frac{(c^2q^2)^{n_E}}{(c^2q^2-(1-u)^2)^{n_E-n_V}} \nn \\
  &\times \det\left(
  \bsone_{n_V} -cq Q_u^{-1} A + c^2 q^2 Q_u^{-1}
  \right)^{-1}\,.
  \label{eq:vertex Bartholdi dual}
\end{align}
\end{lemma}

\begin{proof}
We can prove it by a straightforward calculation 
from the expression \eqref{eq:vertex Bartholdi};
\begin{align*}
  \zeta_G&(1/cq,u) \nn \\
  &=
  \bigl(\bsone_{n_V}-(1-u)^2(cq)^{-2}\bigr)^{-(n_E-n_V)}
  \det\bigl(\bsone_{n_V}-(cq)^{-1} A + (cq)^{-2} Q_u \bigr)^{-1} \nn \\
  &= 
  \mathscale{0.98}{
  \bigl(\bsone_{n_V}-(1-u)^2(cq)^{-2}\bigr)^{-(n_E-n_V)}
  (cq)^{2n_V} \det(Q_u)^{-1}
  \det\left(
  (cq)^2 Q_u^{-1} - cq Q_u^{-1}A + \bsone_{n_V}
  \right)^{-1}} \nn \\
  &=
  \frac{1}{\prod_{v\in V}(1-u)(t_v+u)} 
  \frac{(c^2q^2)^{n_E}}{(c^2q^2-(1-u)^2)^{n_E-n_V}}
  \det\left(
  \bsone_{n_V}-c q Q_u^{-1} A + c^2 q^2 Q_u^{-1}
  \right)^{-1}\,.
\end{align*}
This is the expression \eqref{eq:vertex Bartholdi dual}. 
\end{proof}

Let now $G$ be a $(t+1)$-regular graph, namely, each vertex has the same degree $t+1$.
By setting $c=(1-u)(t+u)$ in \eqref{eq:vertex Bartholdi dual}, we obtain the functional equation of the Bartholdi zeta function \cite{guido2008bartholdi}:
\begin{theorem}[Guido-Isola-Lapidus]
\label{th:funcional eq Bartholdi} 
Suppose $G$ is a $(t+1)$-regular graph and $u\ne 1,-t$. 
Then, $\zeta_G(q,u)$ satisfies the identity, 
as
\begin{align}
  \zeta_G&(1/(1-u)(t+u)q,u)  \nn \\
  &= 
  (-1)^{n_E-n_V}(1-u)^{n_V}(t+u)^{2n_E-n_V}
  q^{2n_E}
  \left(\frac{1-(1-u)^2q^2}{1-(t+u)^2q^2}\right)^{n_E-n_V}
  \zeta_G(q,u)\,.
  \label{eq:functional equation Bartholdi}
\end{align}
\end{theorem}

\begin{proof}
  We can prove it by simply setting $c=(1-u)(t+u)$ in \eqref{eq:vertex Bartholdi dual}.
\end{proof}

The functional equation for the Ihara zeta function is obtained as a corollary of this theorem:
\begin{corollary}
\label{col:funcional eq Ihara}
The Ihara zeta function on a $(t+1)$-regular graph satisfies the identity, 
\begin{align}
  \zeta_G&(1/tq)  
  = 
  (-1)^{n_E-n_V} t^{2n_E-n_V} q^{2n_E}
  \left(\frac{1-q^2}{1-t^2q^2}\right)^{n_E-n_V}
  \zeta_G(q)\,.
  \label{eq:functional equation Ihara}
\end{align}
\end{corollary}

If we define the completed Bartholdi zeta function, 
\begin{equation}
  \xi_G(q,u) \equiv 
  \left(1-(1-u)^2q^2\right)^{n_E-\frac{n_V}{2}}
  \left(1-(t+u)^2q^2\right)^{\frac{n_V}{2}}\zeta_G(q,u)\,,
  \label{eq:completed Bartholdi}
\end{equation}
the functional equation \eqref{eq:functional equation Bartholdi} is expressed in a simpler form, 
\begin{equation}
  \xi_G(1/(1-u)(t+u)q,u) 
  = 
  (-1)^{n_V} \xi_G(q,u)\,.
  \label{eq:functional equation completed Bartholdi}
\end{equation}

If we introduce a parameter $s$ through the relation,
\begin{equation}
  q \equiv |1-u|^{s-1}|t+u|^{-s}
  \qquad \left(\re u \ne 1, -t\right)\,,
  \label{eq:q to s regular}
\end{equation}
the transformation $q\to 1/(1-u)(t+u)q$ is equivalent to $s\to 1-s$, that is, a fold around $\re(s)=1/2$ 
and the functional relation \eqref{eq:functional equation completed Bartholdi} can be written as 
\begin{equation}
  \xi_G(1-s,u) = (-1)^{n_V} \xi_G(s,u)\,.
\end{equation}
This is nothing but an analog of the symmetric functional equation of the Riemann zeta function
up to the sign, 
\begin{equation}
\xi(s)=\xi(1-s)\, ,
\label{eq:symmetric functional equation}
\end{equation}
where $\xi(s)$ is the completed Riemann zeta function, 
\begin{equation}
\xi(s) \equiv \frac{1}{2}\pi^{-\frac{s}{2}}s(s-1)\Gamma\left(
\frac{s}{2}
\right)\zeta(s)\,,
\end{equation}
defined from the Riemann zeta function,
\begin{equation}
  \zeta(s) = \sum_{n=1}^\infty \frac{1}{n^s}\,. 
  \label{eq:Riemann zeta}
  \end{equation}

\subsection{\boldmath The functional equation with respect to the bump parameter $u$}
\label{subsec:u-duality}

We next consider the functional equation of the Bartholdi zeta function with respect to the the bump parameter $u$. 
From the vertex expression of the Bartholdi zeta function on the $(t+1)$-regular graph \eqref{eq:vertex Bartholdi regular}, 
we obtain
\begin{align*}
  \zeta_G(q,1-t-u) 
  &= 
  \left(1-(t+u)^2q^2\right)^{-n_E+n_V}\det\left((1+t+u)(1-u)q^2\bsone_{n_V}-q A \right)^{-1}\,,
\end{align*}
which leads the functional relation with respect to $u$: 
\begin{theorem}
\label{th:funcional eq for u}
Suppose that $G$ is a $(t+1)$-regular graph. 
Then, the Bartholdi zeta function satisfies the identity, 
\begin{equation}
 \zeta_G(q,1-t-u) 
 = \left(\frac{1-(1-u)^2q^2}{1-(t+u)^2q^2}\right)^{n_E-n_V}
   \zeta_G(q,u)\,.
   \label{eq:functional equation Bartholdi wrt u}
\end{equation}
Equivalently, the completed Bartholdi zeta function satisfies 
\begin{equation}
  \xi_G(q,1-t-u) = \xi_G(q,u)\,.
  \label{eq:functional equation completed Bartholdi wrt u}
\end{equation}
\end{theorem}

In particular, it brings a non-trivial equivalence between the Ihara zeta function and the Bartholdi zeta function
by setting $u=0$ in \eqref{eq:functional equation Bartholdi wrt u}:
\begin{corollary}
\label{col:Ihara-Bartholdi duality}
The Ihara zeta function on a $(t+1)$-regular graph has the following relation to the Bartholdi zeta function with the bump parameter $u=1-t$;
\begin{equation}
  \zeta_G(q) = 
  \left(\frac{1-t^2q^2}{1-q^2}\right)^{n_E-n_V}
  \zeta_G(q,1-t)\,.
  \label{eq:Ihara-Bartholdi equivalence}
\end{equation}
\end{corollary}

\section{Poles of the Bartholdi zeta function}
\label{sec:poles}

In this section, we consider pole structure of the Bartholdi zeta function.

\subsection{Critical strip of the Bartholdi zeta function}
\label{subsec:critical strip}


We define the minimal and maximal magnitudes of the poles of the Bartholdi zeta function $\zeta_G(q,u)$ with respect to $q$ as 
$r_G$ and $r'_G$ ($0<r_G\le r'_G$), respectively. 
To this end, we define 
\begin{equation}
  B_u \equiv W + uJ\,,
  \label{eq:matrix Bu}
\end{equation}
where $W$ and $J$ are the edge adjacency matrix and the bump matrix defined in \eqref{eq:matrix W} and \eqref{eq:matrix J}, respectively.
We also define 
\begin{equation}
  \{ t_v | v \in V\} 
  = 
  \{ \tlt_1,\cdots,\tlt_M\}\,, \quad  (\tlt_1 < \cdots < \tlt_M)
\end{equation}
where $M$ is the number of the different vertex degrees of $G$. 
The following proposition determines lower and upper bounds of $r_G$ and $r'_G$, respectively. 

\begin{prop}
\label{prop:range of poles}
$r_G$ and $r'_G$ satisfy the bounds,
\begin{equation}
  R_{\rm min} \le r_G \le r'_G \le R_{\rm max}\,,
\end{equation}
where 
\begin{align}
  R_{\rm min} = \begin{cases}
    |u+\tlt_M|^{-1} & \left( -\frac{\tlt_M - 1}{2} \le \re u \right)\\
    |u-1|^{-1} &  \left(\re u < -\frac{\tlt_M - 1}{2}\right) 
  \end{cases}\,,
  \label{eq:target of proof 1}
\end{align}
and 
\begin{align}
  R_{\rm max} = \begin{cases}
    |u-1|^{-1} &  \left( -\frac{\tlt_1-1}{2} \le \re u \right) \\
    |u+\tlt_i|^{-1} &  \left(-\frac{\tlt_i+\tlt_{i+1}}{2}\le \re u < -\frac{\tlt_{i-1}+\tlt_i}{2}\  \text{for}\ i=1,\cdots,M-1\right) \\
    |u+\tlt_M|^{-1} &  \left(\re u < - \frac{\tlt_{M-1}  + \tlt_M}{2}\right)
  \end{cases}\,.
  \label{eq:target of proof 2}
\end{align}
\end{prop}

\begin{proof}
Since the Bartholdi zeta function is expressed as \eqref{eq:edge Bartholdi}, 
the poles with respect to $q$ are the inverse of the eigenvalues of the matrix $B_u=W+uJ$. 
To examine the range of the eigenvalues of $B_u$, 
it is useful to compute $B_u B_u^\dagger$, 
\begin{align}
  (B_u B_u^\dagger)_{\bse\bse'} = \left(\deg t(\bse) -2 (1-\re(u))\right) \delta_{t(\bse),t(\bse')}
  + |1-u|^2\delta_{\bse,\bse'}\,.
  \label{eq:BBT}
\end{align}
We then define a set of edges whose target is $v$, 
\begin{equation}
  E_{v} \equiv \{\bse\in E_D \;|\; t(\bse)=v\}\,.
\end{equation}
By definition, the set $E_D$ can be expressed as a direct sum of $E_v$ as 
\begin{equation}
  E_D = \bigoplus_{v\in V} E_v\,.
\end{equation}
From \eqref{eq:BBT}, the matrix $B_uB_u^T$ is expressed as a block diagonal matrix as 
\begin{equation}
  B_uB_u^\dagger = \begin{pmatrix}
    F_1 & & \\
    & \ddots & \\
    & & F_{n_V}
  \end{pmatrix}\,, 
  \label{eq:BBdagger}
\end{equation}
where the matrix $F_i$ ($i=1,\cdots,n_V$) is defined by 
\begin{equation}
  F_i \equiv (d_i-2(1-\re(u)))\begin{pmatrix}
    1 & \cdots & 1 \\
    \vdots & \ddots & \vdots \\
    1 & \cdots & 1 
  \end{pmatrix}
  + |1-u|^2 \bsone_{d_i}\,. 
\end{equation}
Since it is easy to show 
\begin{equation}
  \Spec F_i = \left\{|u+d_i-1|^2,|1-u|^2,\cdots,|1-u|^2\right\}\,,
  \label{eq:Spec F}
\end{equation}
the singular values of $B_u$ is given by 
\begin{equation}
  S \equiv \{
  |d_1 + u - 1|, 
  \cdots,
  |d_{n_V} + u - 1|, 
  \overbrace{|1-u|,\cdots,|1-u|}^{2n_{E}-n_V}
  \}\,.
\end{equation}

Using the standard discussion using the Rayleigh quotient, we see that any eigenvalue of $B_u$ satisfies 
\begin{equation}
  \min(S) \le |\lambda| \le \max(S)\,. 
\end{equation}
Since the relation of the magnitudes $|u+x_1|$ and $|u+x_2|$ ($x_1,x_2\in \R$) is the same as that of $|\re(u)+x_1|$ and $|\re(u)+x_2|$, we see 
\begin{align*}
\max(S) = \begin{cases}
  |u+\tlt_M| &  \left(\re u \ge -\frac{\tlt_M - 1}{2}\right)\\
  |u-1| &  \left(\re u < -\frac{\tlt_M - 1}{2}\right)
\end{cases}\,,
\end{align*}
and 
\begin{align*}
\min(S) = \begin{cases}
  |u-1| &  \left(-\frac{\tlt_1-1}{2} \le \re u\right) \\
  |u+\tlt_i| &  \left(-\frac{\tlt_i+\tlt_{i+1}}{2}\le \re u < -\frac{\tlt_{i-1}+\tlt_i}{2}\ \text{for}\ i=1,\cdots,M-1\right) \\
  |u+\tlt_M| &  \left(\re u < - \frac{\tlt_{M-1}  + \tlt_M}{2}\right)
\end{cases}\, ,
\end{align*}
which are equivalent to \eqref{eq:target of proof 1} and \eqref{eq:target of proof 2}, respectively. 
\end{proof}

When $u$ is real and takes a value in a certain region, we can claim that there appear a simple pole on the real axis at the boundary of the critical strip. 
We first show the following lemma: 
\begin{lemma}
  \label{lem:Bu inverse} 
  The inverse of $B_u$ is given by
  \begin{equation}
    B_u^{-1} = \left(\chQ_u^{-1}(W+J) - \frac{1}{1-u}J\right)^T\,,
    \label{eq:Bu inverse}
  \end{equation}
  where $\chQ_u$ is defined in \eqref{eq:matrix Qu}. 
\end{lemma}

\begin{proof}
  We can show it by a direct computation: 
  \begin{align*}
    \biggl(&B_u \left(\chQ_u^{-1}(W+J) - \frac{1}{1-u}J\right)^T \biggr)_{\bse\bsf} \\ 
    &= \sum_{\bse'\in E_D}\left(\delta_{t(\bse)s(\bse')}-(1-u)\delta_{\bse^{-1}\bse'}\right)
    \times \frac{1}{(1-u)(u+t_{s(\bse')})}
    \left(\delta_{s(\bse')t(\bsf)}-(u+t_{s(\bse')})\delta_{\bse'\bsf^{-1}}\right) \\
    &\mathscale{0.95}{= \frac{1}{(1-u)(u+t_{t(\bse)})}
    \left(\deg t(\bse) \delta_{t(\bse)t(\bsf)}-(u+t_{t(\bse)})\delta_{t(\bse)t(\bsf)}\right)
    - \frac{1}{u+t_{t(\bse)}}\left(\delta_{t(\bse)t(\bsf)}-(u+t_{t(\bse)})\delta_{\bse\bsf}\right)} \\
    &=\delta_{\bse\bsf}
  \end{align*}
\end{proof}
Then, we can claim the following. 
\begin{theorem}
  \label{th:first pole at boundary}
  If $u\ge 0$, $\zeta_G(q,u)$ has a simple pole at $q=r_G$. 
  If $u\le -\tlt_M+1$, $\zeta_G(q,u)$ has a simple pole at $q=r'_G$. 
\end{theorem}

\begin{proof}
  If $u\ge 0$, $B_u=W+uJ$ is a non-negative matrix. 
  Therefore, from the Perron-Frobenius theorem, $B_u$ has a positive simple eigenvalue $\lambda_G$ and any other eigenvalue $\lambda$ satisfies $|\lambda|\le \lambda_G$. 
  Since the inverse of the eigenvalues of $B_u$ is the poles of $\zeta_G(q,u)$ from the expression \eqref{eq:edge Bartholdi}, we can conclude $\lambda_G^{-1}$ is on the boundary of the critical strip of $\zeta_G(q,u)$: $q=r_G=\lambda_G^{-1}$ is a simple pole of $\zeta_G(q,u)$ if $u\ge 0$.

  We next consider the inverse of $B_u^{-1}$. 
  From Lemma \ref{lem:Bu inverse}, $B_u^{-1}$ is a non-negative matrix if and only if all the (diagonal) elements of $\chQ_u$ and $\chQ_u^{-1}-\frac{1}{1-u}$ are non-negative, which is achieved when $u \le -t_v+1$ for ${}^{\forall}v\in V$. 
  In this case, from the Perron-Frobenius theorem again, $B_u^{-1}$ has a positive simple eigenvalue $\lambda_G^{\prime -1}$ and any other eigenvalue $\lambda^{-1}$ satisfies $|\lambda^{-1}|\le \lambda_G^{\prime -1}$.
  Since $\lambda'_G$ is an eigenvalue of $B_u$, we can claim that 
  $q= \lambda'_G=r'_G$ is a simple pole of $\zeta_G(q,u)$. 
\end{proof}

Although Proposition \ref{prop:range of poles} determines only bounds of $r_G$ and $r'_G$, 
we can exactly evaluate the values of $r_G$ and $r'_G$ when $G$ is the regular graph: 
\begin{theorem}
\label{th:range of poles regular} 
When $G$ is a $(t+1)$-regular graph, 
both of $q=(t+u)^{-1}$ and $q=(1-u)^{-1}$
are poles of $\zeta_G(q,u)$ 
which saturate the bounds given in Proposition \ref{prop:range of poles}; 
\begin{align}
  (r_G,r'_G)=\begin{cases}
   (|t+u|^{-1},|1-u|^{-1})  & \left(\re u \ge \ds -\frac{t-1}{2}\right) \\
   (|1-u|^{-1},|t+u|^{-1})  & \left(\re u < \ds -\frac{t-1}{2}\right) 
  \end{cases}\,.
\end{align}
\end{theorem}

\begin{proof}
  We can rewrite the vertex expression \eqref{eq:vertex Bartholdi regular} as 
  \begin{equation}
    \zeta_G(q,u) = \left(1-(1-u)^2q^2\right)^{-n_E+n_V}
      \prod_{\lambda\in\Spec A}\left( (1-u)(t+u)q^2 - \lambda q + 1  \right)^{-1}\,.
      \label{eq:tmp4}
  \end{equation}
  Since the adjacency matrix $A$ of the $(t+1)$-regular graph has a simple eigenvalue $\lambda = t+1$ with the eigenvector $(1,\cdots,1)^T$, 
  the solutions of the equation, 
  \begin{equation}
    (1-u)(t+u)q^2 - (t+1) q + 1 = 0\,,
  \end{equation}
  that is, $q=(1-u)^{-1}$ and $q=(t+u)^{-1}$ can be poles of $\zeta_G(q,u)$. 
  Since we are considering a $(t+1)$-regular graph with $t\ge 1$, the number of edges is $n_E=(t+1)n_V/2$. 
  Therefore, the exponent of the factor $\left(1-(1-u)^2q^2\right)^{-(n_E-n_V)}$ 
  of the Bartholdi zeta function \eqref{eq:tmp4} satisfies $n_E-n_V=(t-1)n_V/2\ge 0$, and then both of $q=(1-u)^{-1}$ and $q=(t+u)^{-1}$ are poles of $\zeta_G(q,u)$.
  It is obvious that they saturate the bounds given in Proposition \ref{prop:range of poles}.
\end{proof}

Combining Theorems \ref{th:first pole at boundary} and \ref{th:range of poles regular}, we can claim the following. 
\begin{corollary}
  \label{col:pole at 1/t+u}
  Suppose $G$ is a $(t+1)$-regular graph. 
  When $u\ge 0$ or $u\le -t+1$, $q=(t+u)^{-1}$ is a simple pole of $\zeta_G(q,u)$.
\end{corollary}

We remark that the positive vector $(1,\cdots,1)^{T}$ of size $n_E$ is an eigenvector of $B_u$ with the eigenvalue $t+u$ for the $(t+1)$-regular graph. 
This means that, from the Perron-Frobenius theorem, 
$t+u$ is the maximum eigenvalue of the non-negative matrix $B_u$ when $u\ge 0$ 
and $(t+u)^{-1}$ is the maximum eigenvalue of the non-negative matrix $B_u^{-1}$ when $u\le -t+1$. 
This is why $r_G=(t+u)^{-1}$ for $u\ge 0$ and $r'_G=(t+u)^{-1}$ for $u\le -t+1$. 
A more general condition for $q=(t+u)^{-1}$ to be a simple pole is still an open problem.

When $r_G\ne 0$ nor $r'_G\ne \infty$, we can define a parameter $s$ by 
\begin{equation}
  q \equiv r_G^{s-1}r_G^{\prime\, -s}\,,
  \label{eq:q to s}
\end{equation}
as a natural extension of \eqref{eq:q to s regular} for the regular graph.
Then, the region $r_G\le |q| \le r'_G$ where the poles of $\zeta_G(q,u)$ exist reduces to 
\begin{equation}
  0 \le \re(s) \le 1\,. 
\end{equation}
In particular, when $G$ is $(t+1)$-regular, 
Theorem \ref{th:range of poles regular} guarantees that the non-trivial poles of the Bartholdi zeta function on the regular graph are in the critical strip, 
\begin{equation}
  0 < \re(s) < 1\,, 
\end{equation}
which is resemble to that of the Riemann zeta function. 


Besides the critical strip, 
the Bartholdi zeta function on the regular graph shares similar properties to the Riemann zeta function with respect to the so-called Riemann hypothesis. 
In contrast to that the Riemann hypothesis for the Riemann zeta function is highly non-trivial, we can easily prove that of the Bartholdi zeta function:

\begin{theorem}(Riemann hypothesis for the Bartholdi zeta function)
\label{th:Riemann hypothesis Bartholdi} 
Suppose that $G$ is a $(t+1)$-regular graph and that $u$ is a real number. 
If any eigenvalue $\lambda$ of the adjacency matrix except for $\lambda = t+1$ satisfies 
\begin{equation}
  |\lambda| < 2\sqrt{|(1-u)(t+u)|}\,,
\end{equation}
all poles of the Bartholdi zeta function $\zeta_G(q,u)$ with respect to $q$ 
except for the trivial poles at $q=\pm(1-u)^{-1}$ and $q=(t+u)^{-1}$
are on the critical line $\re(s)=\frac{1}{2}$.
\end{theorem}

\begin{proof}
Using the eigenvalues of the adjacency matrix, 
we can rewrite the expression \eqref{eq:vertex Bartholdi regular} of the Bartholdi zeta function as 
\begin{align}
  \zeta_G(q,u) &= \left(1-(1-u)^2q^2\right)^{-n_E+n_V}
    \prod_{\lambda}\left((1-u)(t+u)q^2- \lambda q + 1\right)^{-1}\,,
\end{align}
where $\lambda$ runs all the eigenvalues of the adjacency matrix $A$. 
Note that $\lambda$ is real since $A$ is a symmetric matrix. 
Therefore, the non-trivial poles of $\zeta_G(q,u)$ corresponding to the eigenvalue $\lambda$ are given by 
\begin{equation}
  q^{(\lambda)}_\pm \equiv \frac{\lambda\pm \sqrt{\lambda^2-4(1-u)(t+u)}}{2(1-u)(t+u)}\,.
\end{equation}
Assuming $\lambda\ne t+1$ which gives poles at $q=(1-u)^{-1}$ and $(t+u)^{-1}$, 
the solutions $q_\pm^{(\lambda)}$ are complex conjugate with each other, 
\begin{equation}
  q_+^{(\lambda)} = q_-^{(\lambda) *}\,,
\end{equation}
since we assume $\lambda^2 < 4(1-u)(t+u)$.
Therefore, the magnitudes of $q_\pm^{(\lambda)}$ is evaluated as 
  $|q_\pm^{(\lambda)}| = |(1-u)(t+u)|^{-\frac{1}{2}}$, 
which is equivalent to $\re(s)=\frac{1}{2}$.
\end{proof}

The Riemann hypothesis of the Ihara zeta function is obtained by simply setting $u=0$ in this theorem:
\begin{corollary}(Riemann hypothesis for the Ihara zeta function)
  Suppose $G$ is a $(t+1)$-regular graph. 
  When any eigenvalue $\lambda$ of the adjacency matrix except for $\lambda = t+1$ satisfies 
  \begin{equation}
    |\lambda| < 2\sqrt{t}\,,
  \end{equation}
  that is, the graph $G$ is Ramanujan, 
  all non-trivial poles of the Ihara zeta function $\zeta_G(q)$ are on the critical line $\re(s)=\frac{1}{2}$. 
\end{corollary}

\subsection{The order of the poles}
\label{subsec:order of the poles}

While we have so far considered the range over which the poles exist, 
we will now examine the common poles of the Bartholdi zeta function at $q=\pm(1-u)^{-1}$ 
that exist independent of the details of the graph.
We start with showing basic properties of the incidence matrix and the $(0,1)$-incidence matrix. 
\begin{lemma}
  \label{lem:rank of incidence}
  The rank of the incidence matrix $L$ of a connected graph $G$ is $n_V-1$. 
\end{lemma}

\begin{proof}
  Suppose $x\in \ker L$. 
  From the definition of the incidence matrix, the elements of $x$ satisfy $x_v = x_{v'}$ when $v$ and $v'$ are connected by an edge. 
  Therefore, since $G$ is connected, 
  all elements of $x\in \ker L$ have the same value: $x=(\alpha,\cdots,\alpha)^T$ for a constant $\alpha$. 
  Apparently, there is no other linearly independent vector in $\ker L$. 
  Thus, the rank of $L$ is $n_V-1$. 
\end{proof}

\begin{lemma}
  \label{lem:L of tree}
  Any $(n_V-1)\times(n_V-1)$ sub-matrix of the incidence matrix $L$ of a tree graph $G$ with $n_V$ vertices is invertible. 
\end{lemma}

\begin{proof}
  Since $G$ is a tree graph, $L$ is an $(n_V-1)\times n_V$ matrix. 
  From the definition of the incidence matrix, if we add all columns from the first to the $(n_V-2)$-th to the $(n_V-1)$'s column, the obtained vector is the minus of the $n_V$-th column of $L$. 
  Therefore, if we assume that the determinant of the sub-matrix made of the first $n_V-1$ columns of $L$ is zero, 
  the determinant of the sub-matrix made of the columns from the first to the $(n_V-2)$-th and $n_V$-th is also zero. 
  In the same manner, we can show that, if one of the $(n_V-1)\times(n_V-1)$ sub-matrices is singular, 
  all the other sum-matrices are also singular. 
  This contradicts to the fact that the rank of $L$ is $n_V-1$. 
  Therefore, any $(n_V-1)\times(n_V-1)$ sub-matrix of $L$ is invertible. 
\end{proof}

\begin{lemma}
  \label{lem:rank of 01-incidence}
  The rank of the $(0,1)$-incidence matrix $\tL$ of a connected graph $G$ is $n_V-1$ if $G$ is bipartite and $n_V$ otherwise. 
\end{lemma}

\begin{proof}
  Suppose $x\in\ker \tL$. 
  From the definition of the $(0,1)$-incidence matrix, the elements of $x$ satisfy $x_v + x_{v'}=0$ when $v$ and $v'$ are connected by an edge. 
  Therefore, 
  since the graph is connected, the absolute values of all elements of $x$ are the same.  
  
  Here, if $G$ has a cycle of odd length made of a permutation of vertices $(v_{i_1},\cdots,v_{i_{2k-1}})$ ($k\in\N$), in other words, if $G$ is not a bipartite graph, the corresponding elements of $x\in \ker \tL$ satisfies 
  \[
  x_{i_1} = -x_{i_2} = \cdots = x_{i_{2k-1}} = -x_{i_1}\,, 
  \]
  which yields $x=0$. 
  Therefore, $\rank{\tL}=n_V$ if $G$ is not bipartite.

  We next consider the case where $G$ is bipartite.
  In this case, we can divide the set of the vertices $V$ into $V = V_1 \cup V_2$ so that any edge connects an element in $V_1$ and an element in $V_2$. 
  We then define an $n_E\times n_V$ matrix $L'$ by 
  \begin{equation}
    L'_{ev} = 
    \begin{cases}
      -\tL_{ev} & (v \in V_1) \\
      \tL_{ev} & (v \in V_2) \\
    \end{cases}\,,
  \end{equation}
  which gives the incidence matrix of a graph $G'$ obtained by switching the directions of some edges of $G$ if it is necessary. 
  Let $T$ be a spanning tree of $G'$. 
  If we restrict the raws of $L'$ to the edges of $T$, we obtain the incidence matrix $L'_T$ of $T$ as a sub-matrix of $L'$. 
  From Lemma \ref{lem:L of tree}, 
  any $n_V-1$ column of $L'_T$ are linearly independent. 
  This means that the corresponding $n_V-1$ column of $\tL$ are linearly independent too. 
  Therefore, the rank of $\tL$ is at least $n_V-1$. 
  If we set $y\in \R^{n_V}$ so that 
  $y_v=-1$ for $v\in V_s$ and 
  $y_v=1$ for $v\in V_t$, 
  $y$ obviously satisfies $\tL y=0$. 
  Thus, we can conclude that the rank of $\tL$ is $n_V-1$. 
\end{proof}

Using these properties, we can claim the followings:
\begin{lemma}
  \label{lem:det sigma at 1-u inv}
  For a connected simple graph $G$ and $u\ne 1$, 
  the matrix $\Sigma=1-qA+q^2 Q_u$ at $q=(1-u)^{-1}$ satisfies 
  \begin{equation}
    \det\Sigma\bigr|_{q=(1-u)^{-1}} = 0\,.
    \label{eq:det Sigma at 1-u inv}
  \end{equation}
\end{lemma}

\begin{proof}
  The matrix $\Sigma$ at $q=(1-u)^{-1}$ reduces to  
  \begin{equation}
    \Sigma\bigr|_{q=(1-u)^{-1}} = \left(1-qA+q^2 Q_u\right)\Bigr|_{q=(1-u)^{-1}} = \frac{\Delta}{1-u}\,,
    \label{eq:Sigma at critical}
  \end{equation} 
  where $\Delta$ is the graph Laplacian defined by \eqref{eq:graph Laplacian}.
  Since $\Delta$ is expressed as $\Delta=L^T L$ by the incidence matrix $L$ and the rank of $L$ is $n_V-1$ from Lemma \ref{lem:rank of incidence}, the rank of $\Delta$ is $n_V-1$ too. 
  Therefore, \eqref{eq:det Sigma at 1-u inv} is satisfied. 
\end{proof}

\begin{lemma}
  \label{lem:det sigma at minus 1-u inv}
  For a connected simple graph $G$ and $u\ne 1$, 
  the matrix $\Sigma=1-qA+q^2 Q_u$ at $q=-(1-u)^{-1}$ satisfies 
  \begin{equation}
    \begin{split}
    \det\Sigma\bigr|_{q=-(1-u)^{-1}} &= 0  \quad  \text{if $G$ is bipartite} \\ 
    \det\Sigma\bigr|_{q=-(1-u)^{-1}} &\ne 0  \quad \text{otherwise} 
    \end{split}
  \label{eq:det Sigma at minus 1-u inv}
  \end{equation}
\end{lemma}

\begin{proof}
  The matrix $\Sigma$ at $q=-(1-u)^{-1}$ reduces to 
  \begin{equation}
    \Sigma\bigr|_{q=-(1-u)^{-1}} = \left(1-qA+q^2 Q_u\right)\Bigr|_{q=-(1-u)^{-1}} = \frac{\tDelta}{1-u}\,,
    \label{eq:Sigma at critical 2}
  \end{equation}
  where $\tDelta$ is defined by \eqref{eq:01incidence and DA}.
  Since $\tDelta$ can be written as $\tDelta = \tL^T \tL$ through the $(0,1)$-incidence matrix $\tL$,
  the rank of $\tDelta$ is $n_V-1$ when $G$ is bipartite and $n_V$ otherwise from Lemma \ref{lem:rank of 01-incidence}.
  Therefore, \eqref{eq:det Sigma at minus 1-u inv} is satisfied. 
\end{proof}

Since the Bartholdi zeta function $\zeta_G(q,u)$ is expressed as \eqref{eq:vertex Bartholdi}, 
it follows from Lemma \ref{lem:det sigma at 1-u inv} and Lemma \ref{lem:det sigma at minus 1-u inv} with respect to the order of the poles at $q=\pm(1-u)^{-1}$ for a connected simple graph $G$ of $n_E\ge n_V$ (see Appendix \ref{app:trivial poles} for another discussion using the Hashimoto expression \eqref{eq:edge Bartholdi}): 
\begin{itemize}
  \item $\zeta_G(q,u)$ has a pole at $q=(1-u)^{-1}$ whose order is at least $n_E-n_V+1$.
  \item When $G$ is bipartite, $\zeta_G(q,u)$ has a pole at $q=-(1-u)^{-1}$ whose order is at least $n_E-n_V+1$.
  \item When $G$ is not bipartite, $q=-(1-u)^{-1}$ is a pole of the order $n_E-n_V$ if $n_E>n_V$ and is not a pole if $n_E=n_V$. 
\end{itemize}

In the following, we determine the order of the poles at $q=\pm(1-u)^{-1}$. 
As mentioned above, for $q=(1-u)^{-1}$, only the minimum order of the pole is determined as $n_E-n_V+1$, but the upper limit has not yet been determined.
On the other hand, for $q=-(1-u)^{-1}$, 
the order of the pole is determined exactly to be $n_E-n_V$ when $G$ is not bipartite. 
Therefore, for $q=-(1-u)^{-1}$, 
we are interested only in the case where $G$ is bipartite. 
However, we can restrict our consideration to $q=(1-u)^{-1}$ 
since the order of $q=-(1-u)^{-1}$ must be the same with that of $q=(1-u)^{-1}$ when $G$ is bipartite from the following lemma:
\begin{lemma}
  \label{lem:even zeta for bipartite}
  The Bartholdi zeta function $\zeta_G(q,u)$ on a bipartite graph $G$ is an even function with respect to $q$. 
\end{lemma}

\begin{proof}
  When $G$ is bipartite, we can separate the set of vertices into $V=V_1 \cup V_2$ so that any edge connects a vertex in $V_1$ and a vertex in $V_2$. 
  Therefore, by arranging the elements of $V$ appropriately, we can express the adjacency matrix as 
  \begin{equation}
    A = \begin{pmatrix}
      0 & B \\ B^T & 0 
    \end{pmatrix}\,,
  \end{equation}
  where $B$ is a $(0,1)$-matrix of size $|V_1|\times |V_2|$. 
  In this setting, the matrix $\Sigma$ can be written as 
  \begin{equation}
    \Sigma = \begin{pmatrix}
      \bsone_{|V_1|} + q^2 Q_{u,1} & -q B \\ -q B^T &  \bsone_{|V_2|} + q^2 Q_{u,2}
    \end{pmatrix}\,,
  \end{equation}
  where
  \[
    Q_{u,a}=\diag((1-u)(\deg v_1^a-1+u),
    \cdots, (1-u)(\deg v_{|V_a|}^a-1+u)),
  \]
  for $a=1,2$ and $v_i^a\in V_a$.
  Since the determinant of $\Sigma$ can be evaluated as 
  \begin{equation}
    \det\Sigma = \det\left( \bsone_{|V_1|} + q^2 Q_{u,1}\right)
    \det\left(
      \bsone_{|V_2|} + q^2 Q_{u,2} 
      + q^2 B^T \left( \bsone_{|V_1|} + q^2 Q_{u,1}\right)^{-1} B
    \right)\,,
  \end{equation}
  $\det\Sigma$ is an even function of $q$. 
  From the expression \eqref{eq:vertex Bartholdi}, $\zeta_G(q,u)$ is an even function of $q$ too.
\end{proof}

Let us consider the behavior of $\det\Sigma$ at $q=(1-u)^{-1}$ in detail. 
From Lemma \ref{lem:det sigma at 1-u inv}, $q=(1-u)^{-1}$ is a zero of $\det\Sigma$ in general. 
To determine the order, we have to estimate the derivative of $\det\Sigma$ with respect to $q$. 
To this end, the following well-known theorem is useful: 
\begin{theorem}[Matrix-Tree theorem]
 \label{th:matrix-tree} 
 The cofactor of any element of the graph Laplacian $\Delta$ is the number of the spanning tree of the graph. 
\end{theorem}
For a proof, see e.g.~\cite{bapat2010graphs}. 
The cofactor is common to all elements of $\Delta$ because it has a unique zero-mode $(1,\cdots,1)^T$. 
In fact, since $\det \Delta=0$, the columns of the classical adjoint matrix of $\Delta$ must be proportional to $(1,\cdots,1)^T$. 
In addition, since $\Delta$ is a symmetric matrix, the classical adjoint matrix of $\Delta$ is also a symmetric matrix. 
Then, the cofactor of $\Delta$ does not depend on the elements.

Let the classical adjoint (adjugate) matrix of $\Sigma$ be denoted by $\adjSigma$.
Using the matrix-tree theorem, we can show the following:
\begin{lemma}
  \label{lem:tSigma at critical}
  When $q=(1-u)^{-1}$, $\adjSigma$ reduces to 
  \begin{equation}
    \adjSigma\big|_{q=(1-u)^{-1}} = \frac{\kappa}{(1-u)^{n_V-1}}\begin{pmatrix}
      1 & \cdots & 1 \\
      \vdots & \ddots & \vdots \\
      1 & \cdots & 1
    \end{pmatrix}\,,
    \label{eq:tSigma at critical}
  \end{equation}
  where $\kappa$ is the number of spanning trees of $G$. 
\end{lemma}

\begin{proof}
  When $q=(1-u)^{-1}$, 
  $\Sigma$ is proportional to the graph Laplacian as \eqref{eq:Sigma at critical}. 
  Therefore, 
  \eqref{eq:tSigma at critical} immediately obeys 
  from the matrix-tree theorem. 
\end{proof}

In the following, we define 
\begin{equation}
  u_*\equiv 1-\frac{\bar{d}}{2}\,,
\end{equation}
where 
\begin{equation}
  \bar{d}\equiv \frac{1}{n_V}\sum_{i=1}^{n_V}d_i = \frac{2n_E}{n_V}\,
\end{equation}
is the average of the degrees of the vertices.

Using Lemma \ref{lem:tSigma at critical}, 
we can show the following lemmas in turn: 
\begin{lemma}
  \label{lem:det prime}
  $\det\Sigma$ is expanded around $q=(1-u)^{-1}$ as 
  \begin{equation}
    \det\Sigma = \delta q  \left( \frac{2 \kappa n_V}{(1-u)^{n_V-1}} \left(u-u_*\right) \right) + {\cal O}(\delta q^2)\,,
    \label{eq:detSigma at critical}
  \end{equation}
\end{lemma}

\begin{proof}
  Since $\det \Sigma = e^{\Tr\log \Sigma}$, the derivative of $\det \Sigma$ by $q$ is expressed as 
  \begin{equation}
    \frac{\partial}{\partial q}\left( \det \Sigma \right)
    = \Tr\left(\adjSigma \Sigma'\right)\,.
  \end{equation}
  $\adjSigma$ and $\Sigma'$ at $q=(1-u)^{-1}$ are given by \eqref{eq:tSigma at critical} and 
  \begin{equation}
    \Sigma'\big|_{q=(1-u)^{-1}} = 2D -A -2(1-u)\bsone_{n_V}\,, 
    \label{eq:Sigma' at critical}
  \end{equation}
  respectively. 
  Then, using $A(1,\cdots,1)^T = D(1,\cdots,1)^T=(d_1,\cdots,d_{n_V})^T$, we obtain 
  \begin{align*}
    \adjSigma \Sigma' \big|_{q=(1-u)^{-1}} 
      = \frac{\kappa}{(1-u)^{n_V-1}}\begin{pmatrix}
      d_1 -2(1-u) & \cdots & d_{n_V} - 2(1-u) \\
      d_1 -2(1-u) & \cdots & d_{n_V} - 2(1-u) \\
      \vdots & \vdots & \vdots \\
      d_1 -2(1-u) & \cdots & d_{n_V} - 2(1-u) 
    \end{pmatrix}\,.
  \end{align*}
  Since $\sum_{i=1}^{n_V} d_i = n_V \bar{d}$ and $\bar{d}-2(1-u) = 2(u-u_*)$ from the definition of $u_*$, we obtain \eqref{eq:detSigma at critical}.
\end{proof}

\begin{lemma}
  \label{lem:Delta_tSigma at critical}
 The product of the graph Laplacian and $\adjSigma'$ at $q=(1-u)^{-1}$ is given by 
 \begin{align}
  \mathscale{0.75}{
  \Delta \adjSigma'\Bigr|_{q=(1-u)^{-1}}
  = \frac{2\kappa}{(1-u)^{n_V-2}}\begin{pmatrix}
    n_V(u-u_*)-(u-u_*^1) & -(u-u_*^1) & \cdots & -(u-u_*^1) \\
    -(u-u_*^2) & n_V(u-u_*)-(u-u_*^2) & \cdots & -(u-u_*^2) \\
    \vdots & \vdots & \ddots & \vdots \\
    -(u-u_*^{n_V}) & -(u-u_*^{n_V}) & \cdots & n_V(u-u_*)-(u-u_*^{n_V})
  \end{pmatrix}
  }
  \label{eq:Delta_tSigma at critical}
 \end{align}
 where $u_*^i \equiv 1-\frac{d_i}{2}$.
\end{lemma}

\begin{proof}
  Since $\adjSigma$ at $q=(1-u)^{-1}$ is given by \eqref{eq:tSigma at critical}, $\adjSigma$ is expanded around $q=(1-u)^{-1}$ as 
  \begin{equation}
    \adjSigma = \frac{\kappa}{(1-u)^{n_V-1}}\begin{pmatrix}
      1 & \cdots & 1 \\
      \vdots & \ddots & \vdots \\
      1 & \cdots & 1
    \end{pmatrix} + \delta q\, \adjSigma'\big|_{q=(1-u)^{-1}} + {\cal O}(\delta q^2)\,,
  \end{equation}
  where $\delta q \equiv q-(1-u)^{-1}$. 
  We can also expand $\Sigma$ as 
  \begin{equation}
    \Sigma = \frac{\Delta}{1-u} + \delta q\,\left(2D-A-2(1-u)\bsone_{n_V}\right) +{\cal O}(\delta q^2)\,. 
  \end{equation}
  Substituting them and the expansion of $\det\Sigma$ \eqref{eq:detSigma at critical} into $\Sigma \adjSigma = \left( \det\Sigma \right) \bsone_{n_V}$, 
  we immediately obtain \eqref{eq:Delta_tSigma at critical}. 
\end{proof}

\begin{lemma}
  \label{lem:MP inverse of Laplacian}
  The Moore-Penrose inverse $\Delta^+$ of the graph Laplacian $\Delta$ satisfies
  \begin{equation}
    \Delta \Delta^+ = \Delta^+ \Delta = \frac{1}{n_V}
    \begin{pmatrix}
      n_V-1 & \cdots & -1 \\
      \vdots & \ddots & \vdots \\
      -1 & \cdots & n_V-1
    \end{pmatrix}\,.
  \end{equation}
\end{lemma}

\begin{proof}
  In general, the Moore-Penrose inverse $M^+$ of a (not necessarily square) matrix $M$ satisfies 
  $M M^+ M = M^+$, 
  $M^+ M M^+ = M$, 
  $(M^+ M)^T = M^+ M$ 
  and 
  $(M M^+)^T = M M^+$.
  Therefore, $\Delta$ and $\Delta^+$ satisfies $\Delta(1-\Delta^+\Delta)=0$. 
  This means that the columns of $1-\Delta^+\Delta$ are proportional to $(1,\cdots,1)^T$. 
  Since $1-\Delta^+\Delta$ is a symmetric matrix, $1-\Delta^+\Delta$ must take the form, 
  \[
    1-\Delta^+\Delta = \alpha 
    \begin{pmatrix}
      1 & \cdots & 1 \\
      \vdots & \ddots & \vdots \\
      1 & \cdots & 1
    \end{pmatrix}\,,
  \]
  with a constant $\alpha$. 
  Furthermore, since $(1-\Delta^+\Delta)^2 = 1-\Delta^+\Delta$, $\alpha$ must be $\frac{1}{n_V}$.
  $\Delta^+\Delta=\Delta\Delta^+$ is a consequence of the property of the Moore-Penrose inverse, 
  $(\Delta^+\Delta)^T = \Delta^+ \Delta$, and the fact that $\Delta$ is a symmetric matrix. 
\end{proof}

\begin{lemma}
  \label{lem:Tr tSigma' Sigma'}
  The trace of the product of $\adjSigma'=\frac{\partial\adjSigma}{\partial q}$ and
  $\Sigma'=\frac{\partial\Sigma}{\partial q}$ at $q=(1-u)^{-1}$ and $u=u_*$ is
  \begin{equation}
    \Tr\left(\adjSigma'\Sigma'\right)\Bigl|_{q=(1-u)^{-1},u=u_*} = \frac{-2\kappa}{(1-u_*)^{n_V-2}}
    \vec{d}\cdot \Delta^+ \vec{d}\,,
    \label{eq:Tr tSigma' Sigma'}
  \end{equation}
  where $\vec{d}\equiv (d_1,\cdots,d_{n_V})^T$.
\end{lemma}

\begin{proof}
  From \eqref{eq:Delta_tSigma at critical}, we obtain 
  \begin{align}
   \Delta \adjSigma'\Bigr|_{q=(1-u)^{-1},u=u_*}
   = \frac{-2\kappa}{(1-u)^{n_V-2}}\begin{pmatrix}
     u_*-u_*^1  & \cdots & u_*-u_*^1 \\
     u_*-u_*^2 & \cdots & u_*-u_*^2 \\
     \vdots  & \ddots & \vdots \\
     u_*-u_*^{n_V} & \cdots & u_*-u_*^{n_V}
   \end{pmatrix}\,,
   \label{eq:Delta_tSigma at critical and u*}
  \end{align}
  by setting $u=u_*$. 
  On the other hand, from Lemma \ref{lem:MP inverse of Laplacian}, 
  the matrix $\Delta\Delta^+$ acts on $\vec{d}$ as 
  \begin{align}
    \Delta \Delta^+ \vec{d} = 
    \begin{pmatrix}
      d_1 - \bar{d} \\ 
      d_2 - \bar{d} \\ 
      \vdots \\
      d_{n_V} - \bar{d} 
    \end{pmatrix}
    = 
    -2 \begin{pmatrix}
      u_*-u^1_* \\ 
      u_*-u^2_* \\ 
      \vdots \\
      u_*-u^{n_V}_* 
    \end{pmatrix}
    \label{eq:tmp1}
  \end{align} 
  Comparing \eqref{eq:Delta_tSigma at critical and u*} and \eqref{eq:tmp1},
  the symmetric matrix $\adjSigma'$ at $q=(1-u)^{-1}$ and $u=u_*$ must take the form, 
  \begin{equation}
    \mathscale{0.85}{
    \adjSigma'\Bigr|_{q=(1-u)^{-1},u=u_*} 
    = 
    \frac{\kappa}{(1-u)^{n_V-2}}\left(
      \Delta^+ 
      \begin{pmatrix}
        d_1 & \cdots & d_{1} \\
        \vdots & \ddots & \vdots \\
        d_{n_V} & \cdots & d_{n_V} 
      \end{pmatrix}
      +
      \begin{pmatrix}
        d_1 & \cdots & d_{n_V} \\
        \vdots & \ddots & \vdots \\
        d_1 & \cdots & d_{n_V} 
      \end{pmatrix}
      \Delta^+
      + \beta 
      \begin{pmatrix}
        1 & \cdots & 1 \\
        \vdots & \ddots & \vdots \\
        1 & \cdots & 1
      \end{pmatrix}
    \right)\,,
    }
    \label{eq:tmp2}
  \end{equation}
  with a constant $\beta$. 
  From \eqref{eq:Sigma' at critical}, $\Sigma'$ at $q=(1-u)^{-1}$ and $u=u_*$ becomes 
  \begin{equation}
    \Sigma'_{*}\equiv \Sigma'\Big|_{q=(1-u)^{-1},u=u_*} = 2D -A -2(1-u_*)\bsone_{n_V}\,, 
  \end{equation}
  which acts to the vectors $(1,\cdots,1)^T$ as 
  \begin{equation}
    \Sigma'_* \begin{pmatrix} 1 \\ \vdots \\ 1 \end{pmatrix} 
    = \begin{pmatrix} d_1 - \bar{d} \\ \vdots \\ d_{n_V} - \bar{d} \end{pmatrix}\,.
  \end{equation}
  Therefore, we can estimate 
  \begin{equation}
    \Sigma'_*
    \begin{pmatrix}
      1 & \cdots & 1 \\
      \vdots & \ddots & \vdots \\
      1 & \cdots & 1
    \end{pmatrix} 
    = 
    \begin{pmatrix}
      d_1-\bar{d} & \cdots & d_1-\bar{d} \\   
      \vdots & \ddots & \vdots \\
      d_{n_V}-\bar{d} & \cdots & d_{n_V}-\bar{d} \\   
    \end{pmatrix}\,, 
  \end{equation}
  and 
  \begin{align*}
    \Sigma'_* 
    &\begin{pmatrix}
      d_1 & \cdots & d_{n_V} \\
      \vdots & \ddots & \vdots \\
      d_1 & \cdots & d_{n_V} 
    \end{pmatrix}
    +
    \begin{pmatrix}
      d_1 & \cdots & d_{n_1} \\
      \vdots & \ddots & \vdots \\
      d_{n_V} & \cdots & d_{n_V} 
    \end{pmatrix}
    \Sigma'_* \\
    &= 
      \begin{pmatrix}
        2d_1^2 -2 \bar{d}d_1 & 2d_1d_2 -\bar{d}(d_1+d_2) & \cdots & 2d_1d_{n_V}-\bar{d}(d_1+d_{n_V}) \\
        2d_2d_1 -\bar{d}(d_2+d_1) & 2d_2^2 -2 \bar{d}d_2 & \cdots & 2d_2d_{n_V}-\bar{d}(d_2+d_{n_V}) \\
        \vdots & \vdots & \ddots & \vdots \\
        2d_{n_V}d_1 -\bar{d}(d_{n_V}+d_1) & 2d_{n_V}d_{2}-\bar{d}(d_{n_V}+d_{2}) & \cdots & 2d_{n_V}^2 -2 \bar{d}d_{n_V} 
      \end{pmatrix}\,.
  \end{align*}
  Substituting them into \eqref{eq:tmp2} and using $\Delta^+ (1,\cdots,1)^T = 0$, 
  we obtain \eqref{eq:Tr tSigma' Sigma'}. 
\end{proof}

We then arrive at the following theorem: 
\begin{theorem}
  \label{th:enhancement}
  Suppose $u\ne 1$. 
  When $G$ is not a tree graph, $q=(1-u)^{-1}$ is a pole of $\zeta_G(q,u)$ whose order is $n_E-n_V+1$ if $u\ne u_*$ 
  while the order of the pole is enhanced to more than or equal to $n_E-n_V+2$ if and only if $u=u_*$.
  When $G$ is a tree graph, $q=(1-u)^{-1}$ is not a pole if $u\ne u_*$ while it becomes a pole if and only if $u=u_*$.
  In particular, if the condition, 
  \begin{equation}
    | (L^+)^T \vec{d} | ^2 \ne n_E\,,
    \label{eq:condition}
  \end{equation}
  is satisfied, 
  the order of the pole at $q=(1-u)^{-1}$ is exactly enhanced to $n_E-n_V+2$ at $u=u_*$. 
\end{theorem}

\begin{proof}
  As a consequence of Lemma \ref{lem:det sigma at 1-u inv} and Lemma \ref{lem:det prime}, 
  $\det\Sigma$ has a simple zero at $q=(1-u)^{-1}$ if $u\ne u_*$. 
  This means that, if $u\ne u_*$, 
  $q=(1-u)^{-1}$ is a pole of $\zeta_G(q,u)$ of order $n_E-n_V+1$ when $G$ is not a tree and is not a pole when $G$ is a tree. 

  On the other hand, 
  from Lemma \ref{lem:det prime}, 
  the order of the zero of $\det\Sigma$ at $q=(1-u)^{-1}$ is enhanced only at $u=u_*$. 
  Therefore, the order of the pole at $q=(1-u)^{-1}$ of $\zeta_G(q,u)$ becomes 
  more than or equal to $n_E-n_V+2$ only when $u=u_*$, 
  which is true even for a tree graph. 

  The order of the pole when $u=u_*$ depends on whether 
  \begin{equation}
    \frac{\partial^2}{\partial q^2}(\det\Sigma) 
    = \Tr(\adjSigma'\Sigma') + \Tr(\adjSigma\Sigma'') 
  \end{equation}
  at $q=(1-u)^{-1}$ and $u=u_*$ is zero or not. 
  Since $\Sigma''$ is given by 
  \begin{equation}
    \Sigma'' = 2(1-u)\diag(u+d_1-1,\cdots,u+d_{n_V}-1)\,,
  \end{equation}
  $\Tr (\adjSigma\Sigma'')$ can be estimated at $q=(1-u)^{-1}$ and $u=u_*$ as 
  \begin{equation}
  \Tr(\adjSigma\Sigma'') \Big|_{q=(1-u)^{-1},u=u_*}
  =  \frac{\kappa}{(1-u_*)^{n_V-2}} n_V \bar{d}
  =  \frac{2\kappa n_E}{(1-u_*)^{n_V-2}} 
  \end{equation}
  by using Lemma \ref{lem:tSigma at critical}.
  By combining this result and Lemma \ref{lem:Tr tSigma' Sigma'}, 
  we can estimate the second derivative of $\det\Sigma$ by $q$ at $q=(1-u)^{-1}$ and $u=u_*$ as 
  \begin{equation}
    \frac{\partial^2}{\partial q^2}(\det\Sigma) \Big|_{q=(1-u)^{-1},u=u_*}
    =  -\frac{2\kappa}{(1-u_*)^{n_V-2}} 
    \left(
    \vec{d}\cdot \Delta^+ \vec{d} - n_E 
    \right)\,.
  \end{equation} 
  Since the graph Laplacian $\Delta$ is expressed as $\Delta=L^TL$, 
  the quantity $\vec{d}\cdot \Delta^+ \vec{d}$ can be written as 
  \begin{equation}
    \vec{d}\cdot \Delta^+ \vec{d}
    = |(L^+)^T \vec{d}|^2\,.
  \end{equation}
  Therefore, if the graph satisfies the condition \eqref{eq:condition}, 
  the order of the pole at $q=(1-u)^{-1}$ of $\zeta_G(q,u)$ is enhanced exactly to $n_E-n_V+2$ when and only when $u=u_*$.
\end{proof}

Combining it with Lemma \ref{lem:det sigma at minus 1-u inv} and Lemma \ref{lem:even zeta for bipartite}, we immediately obtain the following. 
\begin{corollary}
  \label{col:pole at minus 1-u inv} 
  When a graph $G$ is not bipartite and $n_E>n_V$, $q=-(1-u)^{-1}$ is a pole of $\zeta_G(q,u)$ of order $n_E-n_V$. 
  When $G$ is not bipartite and $n_E=n_V$, $q=-(1-u)^{-1}$ is not a pole. 
  When $G$ is bipartite but not a tree, $q=-(1-u)^{-1}$ is a pole of $\zeta_G(q,u)$ 
  whose order is $n_E-n_V+1$ if $u\ne u_*$ and is enhanced to more than or equal to $n_E-n_V+2$ if and only if $u=u_*$.
  When $G$ is a tree graph, $q=-(1-u)^{-1}$ is not a pole if $u\ne u_*$ while it becomes a pole if and only if $u=u_*$.
  In particular, if the graph satisfies 
  $| (L^+)^T \vec{d} | ^2 \ne n_E$, 
  the order is exactly enhanced to $n_E-n_V+2$ at $u=u_*$. 
\end{corollary}

Note that, 
when $G$ is the regular graph, the condition \eqref{eq:condition} is trivially satisfied because $(L^+)^T \vec{d}=0$.
We can understand it as a consequence of the functional equation. 
In fact,
when $G$ is the $(t+1)$-regular graph, 
\[
 u_* = \frac{1-t}{2}\,,
\]
is nothing but the self-dual point of the duality \eqref{eq:functional equation Bartholdi wrt u}. 
As shown in Theorem \ref{th:range of poles regular}, 
$q=(1-u)^{-1}$ and $q=(t+u)^{-1}$ are poles of the Bartholdi zeta function 
and they coincide at $u=u_*$. 
This is why the order of the pole at $q=(1-u)^{-1}$ enhances by one at $u=u_*$
when $G$ is the regular graph. 
Theorem \ref{th:enhancement} indicates that this is not a special phenomenon limited to the regular graph case but rather a general phenomenon of the Bartholdi zeta function. 

We remark that it is still an open problem whether there is a graph that satisfies the equality $| (L^+)^T \vec{d}|^2 = n_E$ and the order of the pole at $q=(1-u)^{-1}$ becomes more than $n_E-n_V+2$.



\section{Discussions and future problems}
\label{sec:discussion}

Since we now have the functional equations and the pole structure of the Bartholdi zeta function, we can examine the nature of the FKM model with $u \ne 0$ by using these properties and repeating the strategy used in \cite{Matsuura:2024gdu}.

We have shown that the functional equation also holds for the bump parameter $u$ such as \eqref{eq:functional equation Bartholdi wrt u},
which gives a non-trivial relationship between the Bartholdi zeta function and the Ihara zeta function as seen in Corollary \ref{col:Ihara-Bartholdi duality}.
Furthermore, 
the Bartholdi zeta function in the large $q$ region can essentially be expressed as the inverse of a polynomial in $1/q$ as shown in \eqref{eq:vertex Bartholdi dual}. 
Similar to the usual Bartholdi zeta function, this polynomial is the characteristic polynomial of a certain matrix, 
but it does not coincide with any known graph zeta function when $G$ is the irregular graph. 
Investigating whether these relationships provide a duality for a wider class of graph zeta functions is worthwhile.

As mentioned in Sec.~\ref{sec:poles}, we have shown that there are two poles at $q=(1-u)^{-1}$ and $q=(t+u)^{-1}$ on the boundary of a critical strip when $G$ is the $(t+1)$-regular graph and the pole at $q=(t+u)^{-1}$ is simple when $u$ satisfies $u\ge 0$ or $u\le -(t-1)$.
This result is particularly important in the context of physics because we usually assume that $q$ and $u$ are real values when we consider the FKM model, and information on the poles of the Bartholdi zeta function is related to the stability of the corresponding FKM model \cite{Matsuura:2024gdu}. 
On the other hand, from a mathematical perspective, how the order of these poles changes when $-(t-1)<u<0$ or when $u$ is complex is an interesting open problem. 
Additionally, there are still many unknowns about the structure of poles when $G$ is an irregular graph, which is also an issue for future research.

\section*{Acknowledgments}
The authors would like to thank 
J.~Fujisawa 
for useful discussions and comments.
This work was supported in part by JSPS KAKENHI Grant Number 23K03423 (K.~O.).


\appendix
\section{Proof of the functional equation via the Hashimoto expression}
\label{app:edge duality}

In this appendix, we provide another proof of the functional equation \eqref{eq:functional equation Bartholdi} of the Bartholdi zeta function $\zeta_G(q,u)$ on a $(t+1)$-regular graph $G$.

We start with showing the identity which holds for a general graph $G$, 
\begin{align}
  \det B_u = \det(W+uJ) 
  = (-1)^{n_E-n_V} (1-u)^{2(n_E-n_V)} \det Q_u \,.
  \label{eq:det Bu}
\end{align}
From \eqref{eq:BBdagger} and \eqref{eq:Spec F}, 
the spectrum of $B_u^\dagger B_u$ is given by 
\[
  {\rm Spec} (B_uB_u^\dagger) = 
  \{
    |d_1 + u - 1|^2, 
    \cdots 
    |d_{n_V} + u - 1|^2, 
    \overbrace{|1-u|^2,\cdots,|1-u|^2}^{2n_{E}-n_V}
    \}\,. 
\]
Therefore, since the elements of $B_u$ are $\pm 1$ or $u$, $\det B_u$ can be written as 
\[
  \det B_u = \epsilon (1-u)^{2n_E-n_V} \prod_{v\in V}(u+t_v)
  = \epsilon (1-u)^{n_E-n_V} \det Q_u\,,
\]
where $\epsilon=\pm 1$. 
To determine $\epsilon$, we note $\lim_{u\to\infty} u^{-2n_E}\det B_u = \epsilon (-1)^{n_V}$. 
Since we can directly evaluate this quantity as 
\[
  \lim_{u\to\infty} u^{-2n_E}\det B_u 
  = \lim_{u\to\infty} \det(u^{-1}W + J)  
  = \det J = (-1)^{n_E}\,,
\]
thus $\epsilon$ is determined to be $\epsilon = (-1)^{n_E-n_V}$. 
Therefore, the identity \eqref{eq:det Bu} holds. 

Then, let us prove the functional equation \eqref{eq:functional equation Bartholdi} from the Hashimoto expression \eqref{eq:edge Bartholdi}. 
The Bartholdi zeta function $\zeta_G(1/(1-u)(t+u)q,u)$ on a $(t+1)$-regular graph can be rewritten as 
\begin{align}
  \zeta_G&(1/(1-u)(t+u)q,u) \nn \\
  &= \det(\bsone_{2n_E}-((1-u)(t+u)q)^{-1} B_u)^{-1} \nn \\
  &= ((1-u)(t+u)q)^{2n_E}(\det B_u)^{-1}\,
  \det(\bsone_{2n_E} - (1-u)(t+u)q B_u^{-1})^{-1}\,. 
  \label{eq:tmp3}
\end{align}
Substituting the inverse of the matrix $B_u$, 
\begin{equation}
  B_u^{-1} = 
  \frac{1}{(1-u)(t+u)}\left(
   W + (1-t-u)J
  \right)\,,
\end{equation}
and $\det B_u$ for the $(t+1)$-regular graph obtained from \eqref{eq:det Bu}, 
\[
\det B_u = (-1)^{n_E-n_V} (1-u)^{2n_E-n_V} (t+u)^{n_V}\,,  
\] 
into \eqref{eq:tmp3},
we can further rewrite it as 
\begin{align*}
  \zeta_G&(1/(1-u)(t+u)q,u) \\
  &= 
  (-1)^{n_E-n_V}(1-u)^{n_V}(t+u)^{2n_E-n_V}
  q^{2n_E}
  \det(\bsone_{2n_E} - q(W+(1-t-u)J))^{-1} \\ 
  &= 
  (-1)^{n_E-n_V}(1-u)^{n_V}(t+u)^{2n_E-n_V}
  q^{2n_E}
  \zeta_G(q,1-t-u) \\ 
  &= 
  (-1)^{n_E-n_V}(1-u)^{n_V}(t+u)^{2n_E-n_V}
  q^{2n_E}
  \frac{(1-(1-u)^2 q^2)^{n_E-n_V}}{(1-(t+u)^2q^2)^{n_E-n_V}} \zeta_G(q,u)\,,
\end{align*}
where 
we have used the functional equation with respect to $u$ \eqref{eq:functional equation Bartholdi wrt u} in the final line. 
This is nothing but the functional equation \eqref{eq:functional equation Bartholdi}.

\section{Analysis of the poles {\boldmath $q=\pm(1-u)^{-1}$} via the Hashimoto expression}
\label{app:trivial poles}

In this appendix, we discuss the order of the poles at $q=\pm(1-u)^{-1}$ by using the Hashimoto expression. 

  As we have mentioned in the proof of Proposition \ref{prop:range of poles}, the poles of the Bartholdi zeta function are given by the inverse of the eigenvalues of $B_u=W+uJ$. 
  It is useful to consider the eigenvectors $\vec{w}^\pm_e$ ($e\in E$) of $J$,  
  \begin{align}
    \left(\vec{w}_e^\pm \right)_{\bsf} \equiv \delta_{e\,\bsf} \pm \delta_{e^{-1}\,\bsf}\,,
  \end{align}
  which satisfy
  \begin{align}
    J \, \vec{w}^{\pm}_e = \pm \vec{w}^\pm_e\,.
  \end{align}
  The edge adjacency matrix $W$ acts on these vectors as 
  \begin{align}
    \begin{split}
    \left(W \vec{w}_e^+ \right)_{\bsf} &= \tL_{e\,t(\bsf)} - \left(\vec{w}^+_e\right)_\bsf\,,\\
    \left(W \vec{w}_e^- \right)_{\bsf} &= L_{e\,t(\bsf)} + \left(\vec{w}^-_e\right)_\bsf\,,\\
    \end{split}
  \end{align}
  where $\tL$ and $L$ 
  are nothing but the $(0,1)$-incidence martix \eqref{eq:01incidence matrix} and the incidence matrix \eqref{eq:incidence matrix}, respectively. 
  This means that, if 
  $\alpha^+ \equiv \left(\alpha^+_1,\cdots,\alpha^+_{n_E}\right)^T$ 
  and
  $\alpha^- \equiv \left(\alpha^-_1,\cdots,\alpha^-_{n_E}\right)^T$ 
  are kernels of $\tL^T$ and $L^T$, respectively, the vectors, 
  \begin{align}
    \vec{w}_{\alpha^\pm} \equiv \sum_{e\in E} \alpha_e^\pm \vec{w}^\pm_e \,,
  \end{align}
  are not only the eigenvectors of $J$ but also those of $W$; 
  \begin{align}
    W \vec{w}_{\alpha^\pm} = \mp \vec{w}_{\alpha^\pm}\,.
  \end{align}
  Thus 
  $\vec{w}_{\alpha^\pm}$ are the eigenvectors of $B_u$ corresponding to the eigenvalues 
  $\mp(1-u)$. 

  Since the rank of the incidence matrix is $n_V-1$ from Lemma \ref{lem:rank of incidence}, 
  $1-u$ is at least $(n_E-n_V+1)$-fold eigenvalue of $W$. 
  Similarly, from Lemma \ref{lem:rank of 01-incidence}, since the rank of the $(0,1)$-incidence matrix is $n_V-1$ for a bipartite graph and $n_V$ for a non-bipartite graph, 
  $-(1-u)$ is an eigenvalue of $B_u$ with multiplicity at least $(n_E-n_V+1)$ if the graph is bipartite and at least $(n_E-n_V)$ if the graph is not bipartite. 

  Comparing this result with Theorem \ref{th:enhancement}, we see that the eigenvectors $\vec{w}_{\pm}$ correspond exactly to the poles at $q=\mp(1-u)^{-1}$ that exist regardless of the value of $u$.

\bibliographystyle{unsrt}
\bibliography{refs}

\end{document}